\documentclass{amsart}
\usepackage{amsfonts}

\newtheorem{theorem}{Theorem}
\theoremstyle{plain}

\numberwithin{equation}{section}
\input{tcilatex}

\begin{document}
\title[On characteristic curves]{On characteristic curves of developable
surfaces in Euclidean 3-space }
\author{Fatih Do\u{g}an}
\address{\textit{Current Adress: Bart\i n University, Department of
Mathematics, 74100, Bart\i n, Turkey}}
\email{\textit{mathfdogan@hotmail.com}}

\begin{abstract}
We investigate the relationship among characteristic curves on developable
surfaces. In case parameter curves coincide with these curves, we show that
the base curve of a developable surface could be either a plane curve, a
circular helix, a general helix or a slant helix.
\end{abstract}

\maketitle

\section{Introduction}

Characteristic curves on surfaces are particular curves such as, geodesic
and asymptotic curves, or lines of curvature. A regular ruled surface in
Euclidean 3-space $\mathbb{E}^{3}$ whose the Gaussian curvature vanishes is
called a developable surface.

Izumiya and Takeuchi $[4]$ studied special curves, like cylindrical helices
and Bertrand curves from the viewpoint of the theory of curves on ruled
surfaces. They enlightened that cylindrical helices are related to Gaussian
curvature and Bertrand curves are related to mean curvature of ruled
surfaces. The same authors $[5]$ defined new special curves that called as
slant helices and conical geodesic curves which are generalizations of the
notion of helices and studied them on developable surfaces. They also
introduced the tangential Darboux developable surface of a space curve which
is defined by the Darboux developable surface of the tangent indicatrix of
the space curve and researched singularities of it.

Leite $[6]$ determined that the orthogonal systems of cycles (curves of
constant geodesic curvature) on the hyperbolic plane $\mathbb{H}^{2}$,
aiming at the classification of maximal surfaces with planar lines of
curvature in Lorentz--Minkowski 3-space $\mathbb{L}^{3}$ and indicated that
a line of curvature on\ the spacelike surface $\mathbb{M}$ is planar if and
only if its normal image in the hyperbolic plane $\mathbb{H}^{2}$ is a
planar curve in $\mathbb{L}^{3}$ as well as a regular curve in $\mathbb{H}%
^{2}$ is planar if and only if it has constant geodesic curvature.

Lucas and Ortega-Yag\"{u}es $[7]$ presented the notion of rectifying curve
in the three-dimensional sphere $\mathbb{S}^{3}(r)$ and denoted that a curve 
$\gamma $ in $\mathbb{S}^{3}(r)$ is a rectifying curve if and only if $%
\gamma $ is a geodesic curve of a conical surface as well as the rectifying
developable surface of a unit speed curve $\gamma $ is a conical surface if
and only if $\gamma $ is a rectifying curve.

Theisel and Farin $[8]$ showed how to compute the curvature and geodesic
curvature of characteristic curves on surfaces, such as contour lines,
reflection lines, lines of curvature, asymptotic curves, and isophote
curves. The conditions of also being characteristic curves of isophotes are
studied in $[2,3]$.

This paper is organized as follows. Section 2 is devoted to some basic
concepts with regard to theory of curves and surfaces in $\mathbb{E}^{3}$
and particular curves on surfaces. In section 3, both the coefficients of
the first and second fundamental forms of developable surfaces and the
Gaussian curvatures of normal and binormal surfaces are obtained. Finally,
in section 4, the main theorems for particular curves on developable
surfaces are given.

\section{Preliminaries}

We shortly give some basic concepts concerning theory of curves and surfaces
in $\mathbb{E}^{3}$ that will use in the subsequent section. Let $\alpha
:I\subset 
\mathbb{R}
\longrightarrow \mathbb{E}^{3}$ be a curve with $\left \Vert \alpha
^{^{\prime }}(s)\right \Vert =1$, where $s$ is the arc-length parameter of $%
\alpha $ and $\alpha ^{^{\prime }}(s)=\dfrac{d\alpha }{ds}(s)$. The function 
$\kappa :I\longrightarrow 
\mathbb{R}
$, $\kappa (s)=\left \Vert \alpha ^{^{\prime \prime }}(s)\right \Vert $ is
defined as the curvature of $\alpha $. For $\kappa >0$, the Frenet frame
along the curve $\alpha $ and the corresponding derivative formulas (Frenet
formulas) are as follows.%
\begin{gather*}
T(s)=\alpha ^{^{\prime }}(s), \\
N(s)=\frac{\alpha ^{^{\prime \prime }}(s)}{\left \Vert \alpha ^{^{\prime
\prime }}(s)\right \Vert }, \\
B(s)=T(s)\times N(s),
\end{gather*}%
where $T$, $N$, and $B$ are the tangent, the principal normal, and the
binormal of $\alpha $, respectively.%
\begin{gather*}
T^{^{\prime }}(s)=\kappa (s)N(s), \\
N^{^{\prime }}(s)=-\kappa (s)T(s)+\tau (s)B(s), \\
B^{^{\prime }}(s)=-\tau (s)N(s),
\end{gather*}%
where the function $\tau :I\longrightarrow 
\mathbb{R}
$, $\tau (s)=\dfrac{\left \langle \alpha ^{^{\prime }}(s)\times \alpha
^{^{\prime \prime }}(s),\alpha ^{^{\prime \prime \prime }}(s)\right \rangle 
}{\kappa ^{2}(s)}$ is the torsion of $\alpha $; "$\left \langle
,\right
\rangle $" is the standart inner product, and "$\times $" is the
cross product on $%
\mathbb{R}
^{3}$.

Let $\mathbb{M}$ be a regular surface and $\alpha :I\subset 
\mathbb{R}
\longrightarrow \mathbb{M}$ be a unit-speed curve. Some particular curves
lying on $\mathbb{M}$ are characterized as follows.\newline
a) A curve $\alpha $ lying on $\mathbb{M}$ is a geodesic curve if and only
if the acceleration vector $\alpha ^{^{\prime \prime }}$ is normal to $%
\mathbb{M}$, in other words,%
\begin{equation*}
U\times \alpha ^{^{\prime \prime }}=0.
\end{equation*}%
b) A curve $\alpha $ lying on $\mathbb{M}$ is an asymptotic curve if and
only if the acceleration vector $\alpha ^{^{\prime \prime }}$ is tangent to $%
\mathbb{M}$, that is,%
\begin{equation*}
\left \langle U,\alpha ^{^{\prime \prime }}\right \rangle =0.
\end{equation*}%
c) A curve $\alpha $ lying on $\mathbb{M}$ is a line of curvature if and
only if $S(T)$ and $T$ are linearly dependent, where $T$ is the tangent of $%
\alpha $, $U$ is the unit normal, and $S$ is the shape operator of $\mathbb{M%
}$.

A ruled surface in $%
\mathbb{R}
^{3}$ is (locally) the map%
\begin{equation*}
F_{(\gamma ,\delta )}(t,u)=\gamma (t)+u\delta (t),
\end{equation*}%
where $\gamma :I\longrightarrow 
\mathbb{R}
^{3}$, $\delta :I\longrightarrow 
\mathbb{R}
^{3}\backslash \{0\}$ are smooth mappings and $I$ is an open interval or a
unit circle $\mathbb{S}^{1}$ $[5]$, where $\gamma $ and $\delta $ are called
the base and generator (director) curves, respectively. For $\left \Vert
\delta (t)\right \Vert =1$, the Gaussian curvature of $F_{(\gamma ,\delta )}$
is $[4]$%
\begin{equation}
\mathbb{K}=\mathbf{-}\frac{[\det (\gamma ^{^{\prime }}(t),\delta (t),\delta
^{^{\prime }}(t))]^{2}}{(EG-F^{2})^{2}},  \tag{2.1}
\end{equation}%
where $E=E(t,u)=\left \Vert \gamma ^{^{\prime }}(t)+u\delta ^{^{\prime
}}(t)\right \Vert ,$ $F=F(t,u)=$ $\left \langle \gamma ^{^{\prime
}}(t),\delta (t)\right \rangle $, and $G=G(t,u)=1$ are the coefficients of
the first fundamental form of $F_{(\gamma ,\delta )}$.

\begin{theorem}[{$[1]$}]
A necessary and sufficient condition for the parameter curves of a surface
to be lines of curvature in a neighborhood of a nonumbilical point is that $%
F=f=0$, where $F$ and $f$ are the respective the first and second
fundamental coefficients.
\end{theorem}

\begin{theorem}[The Lancret Theorem]
Let $\alpha $ be a unit-speed space curve with $\kappa (s)\neq 0$. Then $%
\alpha $ is a general helix if and only if $(\dfrac{\tau }{\kappa })(s)$ is
a constant function.
\end{theorem}

\begin{theorem}[{$[5]$}]
A unit-speed curve $\alpha :I\subset 
\mathbb{R}
\longrightarrow \mathbb{E}^{3}$ with $\kappa (s)\neq 0$ is a slant helix if
and only if%
\begin{equation*}
\sigma (s)=\mp \left( \frac{\kappa ^{2}}{(\kappa ^{2}+\tau ^{2})^{3/2}}(%
\frac{\tau }{\kappa })^{^{\prime }}\right) (s)
\end{equation*}%
is a constant function.
\end{theorem}

\section{Developable, Normal, and Binormal Surfaces}

In this section, we introduce developable, normal and binormal surfaces
associated to a space curve and after that obtain the coefficients of the
first and second fundamental forms of them, respectively. The non-singular
ruled surfaces whose the Gaussian curvature vanish are called developable
surfaces. Now, we firstly get the unit normal $U$ and the acceleration
vectors of the parameter curves of developable surfaces.\newline
Let%
\begin{equation*}
K(s,v)=\alpha (s)+v\delta (s)
\end{equation*}%
be a ruled surface with $\left \Vert \delta (s)\right \Vert =1$. Then, we
have%
\begin{eqnarray*}
K_{s} &=&T+v\delta ^{^{\prime }}, \\
K_{v} &=&\delta ,
\end{eqnarray*}%
\begin{equation*}
U=\frac{K_{s}\times K_{v}}{\left \Vert K_{s}\times K_{v}\right \Vert }=\frac{%
1}{\left \Vert (T+v\delta ^{^{\prime }})\times \delta \right \Vert }%
[(T+v\delta ^{^{\prime }})\times \delta ],
\end{equation*}%
\begin{eqnarray}
K_{vs} &=&\delta ^{^{\prime }},  \notag \\
K_{ss} &=&\kappa N+v\delta ^{^{^{\prime \prime }}},  \notag \\
K_{vv} &=&0,  \TCItag{3.1}
\end{eqnarray}%
where $T$ is the tangent, $N$ is the principal normal, and $\kappa $ is the
curvature of $\alpha $. From Eq.(2.1), the ruled surface $K(s,v)$ is
developable if and only if%
\begin{equation*}
\det (T,\delta ,\delta ^{^{\prime }})=\left \langle T\times \delta ,\delta
^{^{\prime }}\right \rangle =0.
\end{equation*}%
By differentiating the last equation with respect to $s$, we get 
\begin{subequations}
\begin{gather}
\left \langle T^{^{\prime }}\times \delta ,\delta ^{^{\prime }}\right
\rangle +\left \langle T\times \delta ^{^{\prime }},\delta ^{^{\prime
}}\right \rangle +\left \langle T\times \delta ,\delta ^{^{^{\prime \prime
}}}\right \rangle =0  \notag \\
\left \langle T\times \delta ,\delta ^{^{^{\prime \prime }}}\right \rangle
=-\kappa \left \langle N\times \delta ,\delta ^{^{\prime }}\right \rangle . 
\tag{3.2}
\end{gather}%
Secondly, we obtain the coefficients of the first and second fundamental
forms of special developable surfaces.\newline
Let $\alpha :I\subset R\longrightarrow \mathbb{E}^{3}$ be a unit-speed curve
with $\kappa \neq 0$ and let $\{T,N,B,\kappa ,\tau \}$ be Frenet apparatus
of $\alpha $.\newline
\textbf{1)}\textit{\ }The ruled surface 
\end{subequations}
\begin{equation*}
K(s,v)=\alpha (s)+vN(s)
\end{equation*}%
is called the principal normal surface of $\alpha $. The partial derivatives
of $K(s,v)$ with respect to $s$ and $v$ are as follows.%
\begin{gather*}
K_{s}=(1-v\kappa )T+v\tau B, \\
K_{v}=N, \\
K_{ss}=-v\kappa ^{^{\prime }}T+[\kappa -v(\kappa ^{2}+\tau ^{2})]N+v\tau
^{^{\prime }}B, \\
K_{vs}=-\kappa T+\tau B, \\
K_{vv}=0.
\end{gather*}%
Thus, the unit normal $U$ and the Gaussian curvature $\mathbb{K}$ of $K(s,v)$
are obtained as follows.%
\begin{equation*}
U=\frac{K_{s}\times K_{v}}{\left \Vert K_{s}\times K_{v}\right \Vert }=\frac{%
1}{\sqrt{(1-v\kappa )^{2}+v^{2}\tau ^{2}}}[-v\tau T+(1-v\kappa )B],
\end{equation*}%
\begin{subequations}
\begin{align*}
E& =\left \langle K_{s},K_{s}\right \rangle =(1-v\kappa )^{2}+v^{2}\tau ^{2},
\\
F& =\left \langle K_{s},K_{v}\right \rangle =0, \\
G& =\left \langle K_{v},K_{v}\right \rangle =1,
\end{align*}%
\begin{eqnarray}
e &=&\left \langle U,K_{ss}\right \rangle =\frac{v\tau ^{^{\prime
}}+v^{2}\tau ^{2}(\dfrac{\kappa }{\tau })^{^{\prime }}}{\sqrt{(1-v\kappa
)^{2}+v^{2}\tau ^{2}}},  \TCItag{3.3} \\
f &=&\left \langle U,K_{vs}\right \rangle =\frac{\tau }{\sqrt{(1-v\kappa
)^{2}+v^{2}\tau ^{2}}},  \notag \\
g &=&\left \langle U,K_{vv}\right \rangle =0,  \notag
\end{eqnarray}%
\end{subequations}
\begin{equation*}
\mathbb{K}=\frac{eg-f^{2}}{EG-F^{2}}=-\frac{\tau ^{2}}{[(1-v\kappa
)^{2}+v^{2}\tau ^{2}]^{2}},
\end{equation*}%
\textbf{2)}\textit{\ }The ruled surface%
\begin{equation*}
K(s,v)=\alpha (s)+vB(s)
\end{equation*}%
is called the binormal surface of $\alpha $. The partial derivatives of $%
K(s,v)$ with respect to $s$ and $v$ are as follows.%
\begin{gather*}
K_{s}=T-v\tau N, \\
K_{v}=B, \\
K_{ss}=v\kappa \tau T+(\kappa -v\tau ^{^{\prime }})N-v\tau ^{2}B, \\
K_{vs}=-\tau N, \\
K_{vv}=0.
\end{gather*}%
Thus, the unit normal $U$ and the Gaussian curvature $\mathbb{K}$ of $K(s,v)$
are obtained as follows.%
\begin{equation*}
U=\frac{K_{s}\times K_{v}}{\left \Vert K_{s}\times K_{v}\right \Vert }=\frac{%
1}{\sqrt{1+v^{2}\tau ^{2}}}[-v\tau T-N],
\end{equation*}%
\begin{eqnarray*}
E &=&\left \langle K_{s},K_{s}\right \rangle =1+v^{2}\tau ^{2}, \\
F &=&\left \langle K_{s},K_{v}\right \rangle =0, \\
G &=&\left \langle K_{v},K_{v}\right \rangle =1,
\end{eqnarray*}%
\begin{eqnarray}
e &=&\left \langle U,K_{ss}\right \rangle =\frac{-v^{2}\kappa \tau
^{2}-\kappa +v\tau ^{^{\prime }}}{\sqrt{1+v^{2}\tau ^{2}}},  \TCItag{3.4} \\
f &=&\left \langle U,K_{vs}\right \rangle =\frac{\tau }{\sqrt{1+v^{2}\tau
^{2}}},  \notag \\
g &=&\left \langle U,K_{vv}\right \rangle =0,  \notag
\end{eqnarray}%
\begin{equation*}
\mathbb{K}=\frac{eg-f^{2}}{EG-F^{2}}=-\frac{\tau ^{2}}{[1+v^{2}\tau ^{2}]^{2}%
}.
\end{equation*}%
As we can see above, the normal and binormal surfaces of $\alpha $ are
developable if and only if the base curve $\alpha $ is a plane curve.\newline
\textbf{3)}\textit{\ }The ruled surface%
\begin{equation*}
K(s,v)=\alpha (s)+vT(s)
\end{equation*}%
is called the tangent developable surface of $\alpha $. The partial
derivatives of $K(s,v)$ with respect to $s$ and $v$ are as follows.%
\begin{gather*}
K_{s}=T+v\kappa N, \\
K_{v}=T, \\
K_{ss}=-v\kappa ^{2}T+(\kappa +v\kappa ^{^{\prime }})N+v\kappa \tau B, \\
K_{vs}=\kappa N, \\
K_{vv}=0.
\end{gather*}%
Thus, the unit normal $U$ and the coefficients of the first and second
fundamental forms of $K(s,v)$ are obtained as follows.%
\begin{equation*}
U=\frac{K_{s}\times K_{v}}{\left \Vert K_{s}\times K_{v}\right \Vert }=\pm B,
\end{equation*}%
\begin{eqnarray*}
E &=&\left \langle K_{s},K_{s}\right \rangle =1+v^{2}\kappa ^{2}, \\
F &=&\left \langle K_{s},K_{v}\right \rangle =1, \\
G &=&\left \langle K_{v},K_{v}\right \rangle =1,
\end{eqnarray*}%
\begin{eqnarray}
e &=&\left \langle U,K_{ss}\right \rangle =\pm v\kappa \tau ,  \TCItag{3.5}
\\
f &=&\left \langle U,K_{vs}\right \rangle =0,  \notag \\
g &=&\left \langle U,K_{vv}\right \rangle =0.  \notag
\end{eqnarray}%
\textbf{4)}\textit{\ }The ruled surface%
\begin{equation*}
K(s,v)=B(s)+vT(s)
\end{equation*}%
is called the Darboux developable surface of $\alpha $. The partial
derivatives of $K(s,v)$ with respect to $s$ and $v$ are as follows.%
\begin{gather*}
K_{s}=(v\kappa -\tau )N, \\
K_{v}=T, \\
K_{ss}=-\kappa (v\kappa -\tau )T+(v\kappa -\tau )^{^{\prime }}N+\tau
(v\kappa -\tau )B, \\
K_{vs}=\kappa N, \\
K_{vv}=0.
\end{gather*}%
Thus, the unit normal $U$ and the coefficients of the first and second
fundamental forms of $K(s,v)$ are obtained as follows.%
\begin{equation*}
U=\frac{K_{s}\times K_{v}}{\left \Vert K_{s}\times K_{v}\right \Vert }=\pm B,
\end{equation*}%
\begin{eqnarray*}
E &=&\left \langle K_{s},K_{s}\right \rangle =(v\kappa -\tau )^{2}, \\
F &=&\left \langle K_{s},K_{v}\right \rangle =0, \\
G &=&\left \langle K_{v},K_{v}\right \rangle =1,
\end{eqnarray*}%
\begin{eqnarray}
e &=&\left \langle U,K_{ss}\right \rangle =\pm \tau (v\kappa -\tau ), 
\TCItag{3.6} \\
f &=&\left \langle U,K_{vs}\right \rangle =0,  \notag \\
g &=&\left \langle U,K_{vv}\right \rangle =0.  \notag
\end{eqnarray}%
\textbf{5)} The ruled surface%
\begin{equation*}
K(s,v)=\alpha (s)+v\overset{\sim }{D}(s)
\end{equation*}%
is called the rectifying developable surface of $\alpha $, where%
\begin{equation*}
\overset{\sim }{D}(s)=(\dfrac{\tau }{\kappa })(s)T(s)+B(s)
\end{equation*}%
is the modified Darboux vector field of $\alpha $. The partial derivatives
of $K(s,v)$ with respect to $s$ and $v$ are as follows.%
\begin{gather*}
K_{s}=(1+v(\dfrac{\tau }{\kappa })^{^{\prime }})T, \\
K_{v}=\dfrac{\tau }{\kappa }T+B, \\
K_{ss}=v(\dfrac{\tau }{\kappa })^{^{^{\prime \prime }}}T+\kappa (1+v(\dfrac{%
\tau }{\kappa })^{^{\prime }})N, \\
K_{vs}=(\dfrac{\tau }{\kappa })^{^{\prime }}T, \\
K_{vv}=0.
\end{gather*}%
Thus, the unit normal $U$ and the coefficients of the first and second
fundamental forms of $K(s,v)$ are obtained as follows.%
\begin{equation*}
U=\frac{K_{s}\times K_{v}}{\left \Vert K_{s}\times K_{v}\right \Vert }=\pm N,
\end{equation*}%
\begin{eqnarray*}
E &=&\left \langle K_{s},K_{s}\right \rangle =(1+v(\dfrac{\tau }{\kappa }%
)^{^{\prime }})^{2}, \\
F &=&\left \langle K_{s},K_{v}\right \rangle =\dfrac{\tau }{\kappa }(1+v(%
\dfrac{\tau }{\kappa })^{^{\prime }}), \\
G &=&\left \langle K_{v},K_{v}\right \rangle =1+\dfrac{\tau ^{2}}{\kappa ^{2}%
},
\end{eqnarray*}%
\begin{eqnarray}
e &=&\left \langle U,K_{ss}\right \rangle =\pm \kappa (1+v(\dfrac{\tau }{%
\kappa })^{^{\prime }}),  \TCItag{3.7} \\
f &=&\left \langle U,K_{vs}\right \rangle =0,  \notag \\
g &=&\left \langle U,K_{vv}\right \rangle =0.  \notag
\end{eqnarray}%
\textbf{6)} The ruled surface%
\begin{equation*}
K(s,v)=\overset{-}{D}(s)+vN(s)
\end{equation*}%
is called the tangential Darboux developable surface of $\alpha $, where%
\begin{equation*}
\overset{-}{D}(s)=\dfrac{1}{\sqrt{\kappa ^{2}(s)+\tau ^{2}(s)}}(\tau
(s)T(s)+\kappa (s)B(s))
\end{equation*}%
is the unit Darboux vector field of $\alpha $. The partial derivatives of $%
K(s,v)$ with respect to $s$ and $v$ are as follows.%
\begin{gather*}
K_{s}=(v-\sigma (s))N^{^{\prime }}, \\
K_{v}=N, \\
K_{ss}=-\sigma ^{^{\prime }}(s)N^{^{\prime }}+(v-\sigma (s))N^{^{\prime
\prime }}, \\
K_{vs}=N^{^{\prime }}, \\
K_{vv}=0,
\end{gather*}%
where $N^{^{\prime \prime }}=-\kappa ^{^{\prime }}T-(\kappa ^{2}+\tau
^{2})N+\tau ^{^{\prime }}B$. Thus, the unit normal $U$ and the coefficients
of the first and second fundamental forms of $K(s,v)$ are obtained as
follows.%
\begin{equation*}
U=\frac{K_{s}\times K_{v}}{\left \Vert K_{s}\times K_{v}\right \Vert }=\pm 
\overset{-}{D}
\end{equation*}%
\begin{gather*}
E=\left \langle K_{s},K_{s}\right \rangle =(v-\sigma (s))^{2}(\kappa
^{2}+\tau ^{2}), \\
F=\left \langle K_{s},K_{v}\right \rangle =0, \\
G=\left \langle K_{v},K_{v}\right \rangle =1,
\end{gather*}%
\begin{gather}
e=\left \langle U,K_{ss}\right \rangle =\pm (\kappa ^{2}+\tau ^{2})\sigma
(s)(v-\sigma (s)),  \tag{3.8} \\
f=\left \langle U,K_{vs}\right \rangle =0,  \notag \\
g=\left \langle U,K_{vv}\right \rangle =0.  \notag
\end{gather}%
From now on, we will investigate the relationship among characteristic
curves of developable surfaces.

\section{Main Results}

We give main theorems that characterize the parameter curves are also
particular curves on $K(s,v)$ such as geodesic curves, asymptotic curves or
lines of curvature.

\begin{theorem}
Let $K(s,v)=\alpha (s)+v\delta (s)$\ be a developable surface with $%
\left
\Vert \delta (s)\right \Vert =1$. Then, the following expressions are
satisfied for the parameter curves of $K(s,v)$.\newline
i) $s-$parameter curves are also asymptotic curves if and only if%
\begin{equation*}
\frac{\det (\delta ,\delta ^{^{\prime }},\delta ^{^{^{\prime \prime }}})}{%
\det (T,\delta ,N)}=\frac{\kappa }{v^{2}},
\end{equation*}%
where $T$ is the tangent, $N$ is the principal normal, and $\kappa $ is the
curvature of $\alpha $.\newline
ii) $s-$parameter curves cannot also be geodesic curves.\newline
iii) $v-$parameter curves are straight lines.\newline
iv) The parameter curves are also lines of curvature if and only if the
tangent of the base curve $\alpha $ is perpendicular to the director curve $%
\delta $.
\end{theorem}

\begin{proof}
i) If we use Eq.(3.1) and Eq.(3.2), and then take inner product of $U$ and $%
K_{ss}$, we obtain%
\begin{equation*}
\left \langle U,K_{ss}\right \rangle =\kappa \left \langle T\times \delta
,N\right \rangle +\kappa v\left \langle \delta ^{^{\prime }}\times \delta
,N\right \rangle +v\left \langle T\times \delta ,\delta ^{^{^{\prime \prime
}}}\right \rangle +v^{2}\left \langle \delta ^{^{\prime }}\times \delta
,\delta ^{^{^{\prime \prime }}}\right \rangle =0
\end{equation*}%
\begin{equation*}
\frac{\det (\delta ,\delta ^{^{\prime }},\delta ^{^{^{\prime \prime }}})}{%
\det (T,\delta ,N)}=\frac{\kappa }{v^{2}}.
\end{equation*}%
ii) If we take cross product of $U$ and $K_{ss}$, we get%
\begin{eqnarray*}
U\times K_{ss} &=&-[\kappa \left \langle \delta ,N\right \rangle +v\left
\langle \delta ,\delta ^{^{^{\prime \prime }}}\right \rangle ]T+[\kappa
v\left \langle \delta ^{^{\prime }},N\right \rangle +v\left \langle T,\delta
^{^{^{\prime \prime }}}\right \rangle +v^{2}\left \langle \delta ^{^{\prime
}},\delta ^{^{^{^{\prime \prime }}}}\right \rangle ]\delta \\
&&-[\kappa v\left \langle \delta ,N\right \rangle +v^{2}\left \langle \delta
,\delta ^{^{^{\prime \prime }}}\right \rangle ]\delta ^{^{^{\prime }}}.
\end{eqnarray*}%
Since $T$, $\delta $, and $\delta ^{^{\prime }}$are linearly dependent, the
coefficients of them cannot be zero at the same time, i.e, $s-$parameter
curves cannot also be geodesic curves.

iii) \textit{Since }$U\times K_{vv}=0$ and $\left \langle
U,K_{vv}\right
\rangle =0$, $v-$parameter curves are both geodesic curves
and asymptotic curves, namely, $v-$parameter curves are straight lines.

iv) Since $K(s,v)$ is the developable surface, from Eq.(3.1), we obtain%
\begin{eqnarray*}
f &=&\left \langle U,K_{vs}\right \rangle =\left \langle T\times \delta
,\delta ^{^{\prime }}\right \rangle +v\left \langle \delta ^{^{\prime
}}\times \delta ,\delta ^{^{\prime }}\right \rangle =0, \\
F &=&\left \langle K_{s},K_{v}\right \rangle =\left \langle T,\delta \right
\rangle +v\left \langle \delta ,\delta ^{^{\prime }}\right \rangle .
\end{eqnarray*}%
By the last equation, the parameter curves are also lines of curvature if
and only if $\left \langle T,\delta \right \rangle =0$, that is, the tangent
of the base curve $\alpha $ is perpendicular to the director curve $\delta $.
\end{proof}

\begin{theorem}
Let $K(s,v)$ be the normal surface of $\alpha $. Then, the following
expressions are satisfied for the parameter curves of $K(s,v)$.\newline
\textit{i)} $s-$parameter curves are also asymptotic curves if and only if
the exactly one $s_{0}-$ parameter curve is the base curve, i.e., $%
K(s,0)=\alpha (s)$ or the base curve $\alpha $ is a circular helix.\newline
ii) $s-$parameter curves are also geodesic curves if and only if the base
curve $\alpha $ is a circular helix.\newline
iii) $v-$parameter curves are straight lines.\newline
iv) The parameter curves are also lines of curvature if and only if the base
curve $\alpha $ is a plane curve.
\end{theorem}

\begin{proof}
\textit{i)} From Eq.(3.3), we have%
\begin{equation*}
\left \langle U,K_{ss}\right \rangle =0\Longleftrightarrow v\tau ^{^{\prime
}}+v^{2}\tau ^{2}(\dfrac{\kappa }{\tau })^{^{\prime }}=0
\end{equation*}%
\begin{equation*}
v\tau ^{^{\prime }}+v^{2}\tau ^{2}(\dfrac{\kappa }{\tau })^{^{\prime }}=0
\end{equation*}%
\begin{eqnarray*}
v &=&0\text{ \  \  \ or \  \  \ }(\dfrac{1}{\tau })^{^{\prime }}=v(\dfrac{\kappa 
}{\tau })^{^{\prime }} \\
v &=&0\text{ \  \  \ or \  \  \ }\kappa =\dfrac{1-c\tau }{v},
\end{eqnarray*}%
where $c$ is a constant. Then, the $s_{0}-$ parameter curve is the base
curve $K(s,0)=\alpha (s)$ or $\kappa $ and $\tau $ are constants. As a
result, the base curve $\alpha $ becomes a circular helix.\newline
\textit{ii) }$s-$parameter curves are also geodesic curves if and only if%
\begin{equation*}
U\times K_{ss}=\frac{1}{\sqrt{(1-v\kappa )^{2}+v^{2}\tau ^{2}}}\left[ 
\begin{array}{c}
(\kappa -v(\kappa ^{2}+\tau ^{2}))(1-v\kappa )T \\ 
+(v\kappa ^{^{\prime }}-\dfrac{v^{2}}{2}(\kappa ^{2}+\tau ^{2})^{^{\prime
}})N \\ 
+v\tau (\kappa -v(\kappa ^{2}+\tau ^{2}))B%
\end{array}%
\right] =0.
\end{equation*}%
Since $T$, $N$, and $B$ are linearly independent, we have%
\begin{eqnarray*}
(\kappa -v(\kappa ^{2}+\tau ^{2}))(1-v\kappa ) &=&0, \\
v\kappa ^{^{\prime }}-\frac{v^{2}}{2}(\kappa ^{2}+\tau ^{2})^{^{\prime }}
&=&0, \\
v\tau (\kappa -v(\kappa ^{2}+\tau ^{2})) &=&0.
\end{eqnarray*}%
From the expression of $U$, we see that both $1-v\kappa $ and $v\tau $
cannot be zero at the same time. Hence, by the first and last equations
above, we obtain $\kappa =v(\kappa ^{2}+\tau ^{2})$. If we substitute this
in the second equation, we get $\dfrac{v^{2}}{2}(\kappa ^{2}+\tau
^{2})^{^{\prime }}=0$. As $\kappa \neq 0$, $v\neq 0$. Then we have $\kappa
^{2}+\tau ^{2}=$ $constant$. Since$\ v$ and $\kappa ^{2}+\tau ^{2}$ are
constants, from $\kappa =v(\kappa ^{2}+\tau ^{2})$, it follows that $\kappa $
is a constant and thus $\tau $ is a constant. In other words, the base curve 
$\alpha $ is a circular helix.\newline
\textit{iii) Since }$U\times K_{vv}=0$ and $\left \langle
U,K_{vv}\right
\rangle =0$, $v-$parameter curves are both geodesic curves
and asymptotic curves, namely, $v-$parameter curves are straight lines.%
\newline
\textit{iv) We have }$F=0$. Moreover,%
\begin{equation*}
f=\dfrac{\tau }{\sqrt{(1-v\kappa )^{2}+v^{2}\tau ^{2}}}=0\Longleftrightarrow
\tau =0.
\end{equation*}%
Then, the parameter curves are also lines of curvature if and only if the
base curve $\alpha $ is a plane curve.
\end{proof}

\begin{theorem}
Let $K(s,v)$ be the binormal surface of $\alpha $. Then, the following
expressions are satisfied for the parameter curves of $K(s,v)$.\newline
\textit{i)} $s-$parameter curves are also asymptotic curves if and only the
curvatures of $\alpha $ hold the differential equation%
\begin{equation*}
v\tau ^{^{\prime }}-v^{2}\kappa \tau ^{2}-\kappa =0.
\end{equation*}%
ii) $s-$parameter curves are also geodesic curves if and only if the exactly
one $s_{0}-$ parameter curve is the base curve, i.e., $K(s,0)=\alpha (s)$ or
the base curve $\alpha $ is a plane curve.\newline
iii) $v-$parameter curves are straight lines.\newline
iv) The parameter curves are also lines of curvature if and only if the base
curve $\alpha $ is a plane curve.
\end{theorem}

\begin{proof}
\textit{i)} From Eq.(3.4), we have%
\begin{equation*}
\left \langle U,K_{ss}\right \rangle =0\Longleftrightarrow -v^{2}\kappa \tau
^{2}-\kappa +v\tau ^{^{\prime }}=0.
\end{equation*}%
\textit{ii) }$s-$parameter curves are also geodesic curves if and only if%
\begin{equation*}
U\times K_{ss}=\frac{1}{\sqrt{1+v^{2}\tau ^{2}}}[v\tau ^{2}T-v^{2}\tau
^{3}N+v^{2}\tau \tau ^{^{\prime }}B]=0.
\end{equation*}%
Since $T$, $N$, and $B$ are linearly independent, we have%
\begin{eqnarray*}
v\tau ^{2} &=&0, \\
v^{2}\tau ^{3} &=&0, \\
v^{2}\tau \tau ^{^{\prime }} &=&0.
\end{eqnarray*}%
Then, it follows that $v=0$ or $\tau =0$, i.e., the $s_{0}-$ parameter curve
is the base curve $K(s,0)=\alpha (s)$ or the base curve $\alpha $ is a plane
curve.\newline
\textit{iii) Since }$U\times K_{vv}=0$ and $\left \langle
U,K_{vv}\right
\rangle =0$, $v-$parameter curves are both geodesic curves
and asymptotic curves, namely, $v-$parameter curves are straight lines.%
\newline
\textit{iv) We have }$F=0$. Moreover,%
\begin{equation*}
f=\frac{\tau }{\sqrt{1+v^{2}\tau ^{2}}}=0\Longleftrightarrow \tau =0.
\end{equation*}%
Then, the parameter curves are also lines of curvature if and only if the
base curve $\alpha $ is a plane curve.
\end{proof}

\begin{theorem}
Let $K(s,v)$ be the tangent developable surface of $\alpha $. Then, the
following expressions are satisfied for the parameter curves of $K(s,v)$.%
\newline
\textit{i)} $s-$parameter curves are also asymptotic curves if and only if
the exactly one $s_{0}-$ parameter curve is the base curve, i.e., $%
K(s,0)=\alpha (s)$ or the base curve $\alpha $ is a plane curve.\newline
ii) $s-$parameter curves are also geodesic curves if and only if the exactly
one $s_{0}-$ parameter curve is the base curve, i.e., $K(s,0)=\alpha (s)$.%
\newline
iii) $v-$parameter curves are straight lines.\newline
iv) The parameter curves cannot also be lines of curvature.
\end{theorem}

\begin{proof}
\textit{i)} From Eq.(3.5), we have%
\begin{equation*}
\left \langle U,K_{ss}\right \rangle =0\Longleftrightarrow \pm v\kappa \tau
=0.
\end{equation*}%
Since $\kappa \neq 0$, $v=0$ or $\tau =0$, i.e., the $s_{0}-$ parameter
curve is the base curve $K(s,0)=\alpha (s)$ or the base curve $\alpha $ is a
plane curve.\newline
\textit{ii) }$s-$parameter curves are also geodesic curves if and only if%
\begin{equation*}
U\times K_{ss}=\mp (\kappa +v\kappa ^{^{\prime }})T\mp v\kappa ^{2}N=0.
\end{equation*}%
Since $T$ and $N$ are linearly independent, we have%
\begin{eqnarray*}
\kappa +v\kappa ^{^{\prime }} &=&0, \\
v\kappa ^{2} &=&0.
\end{eqnarray*}%
Since $\kappa \neq 0$, it follows that $v=0$, i.e., the $s_{0}-$ parameter
curve is the base curve $K(s,0)=\alpha (s)$.\newline
\textit{iii) Since }$U\times K_{vv}=0$ and $\left \langle
U,K_{vv}\right
\rangle =0$, $v-$parameter curves are both geodesic curves
and asymptotic curves, namely, $v-$parameter curves are straight lines.%
\newline
\textit{iv) We have }$F=0$ and $f=1$. Therefore, the parameter curves cannot
also be lines of curvature.
\end{proof}

\begin{theorem}
Let $K(s,v)$ be the Darboux developable surface of $\alpha $. Then, the
following expressions are satisfied for the parameter curves of $K(s,v)$.%
\newline
\textit{i)} $s-$parameter curves are also asymptotic curves if and only if
the base curve $\alpha $ is a plane curve or a general helix.\newline
ii) $s-$parameter curves are also geodesic curves if and only if the base
curve $\alpha $ is a general helix.\newline
iii) $v-$parameter curves are straight lines.\newline
iv) The parameter curves are also lines of curvature.
\end{theorem}

\begin{proof}
\textit{i)} From Eq.(3.6), we have%
\begin{equation*}
\left \langle U,K_{ss}\right \rangle =0\Longleftrightarrow \pm \tau (v\kappa
-\tau )=0.
\end{equation*}%
Then, $\tau =0$ or $\dfrac{\tau }{\kappa }=v=$ $constant$, i.e., the base
curve $\alpha $ is a plane curve or a general helix.\newline
\textit{ii) }$s-$parameter curves are also geodesic curves if and only if%
\begin{equation*}
U\times K_{ss}=\mp (v\kappa -\tau )^{^{\prime }}T\mp \kappa (v\kappa -\tau
)N=0.
\end{equation*}%
Since $T$ and $N$ are linearly independent, we have%
\begin{eqnarray*}
(v\kappa -\tau )^{^{\prime }} &=&0, \\
\kappa (v\kappa -\tau ) &=&0.
\end{eqnarray*}%
Then, it follows that $\dfrac{\tau }{\kappa }=v=$ $constant$, i.e., the base
curve $\alpha $ is a general helix.\newline
\textit{iii) Since }$U\times K_{vv}=0$ and $\left \langle
U,K_{vv}\right
\rangle =0$, $v-$parameter curves are both geodesic curves
and asymptotic curves, namely, $v-$parameter curves are straight lines.%
\newline
\textit{iv) We have }$F=f=0$. Consequently, the parameter curves are also
lines of curvature.
\end{proof}

\begin{theorem}
Let $K(s,v)$ be the rectifying developable surface of $\alpha $. Then, the
following expressions are satisfied for the parameter curves of $K(s,v)$.%
\newline
\textit{i)} $s-$parameter curves are also asymptotic curves if and only if $%
\dfrac{\tau }{\kappa }$ is a linear function.\newline
ii) $s-$parameter curves are also geodesic curves if and only if the exactly
one $s_{0}-$ parameter curve is the base curve, i.e., $K(s,0)=\alpha (s)$ or 
$\dfrac{\tau }{\kappa }$ is a linear function.\newline
iii) $v-$parameter curves are straight lines.\newline
iv) The parameter curves are also lines of curvature if and only if the base
curve $\alpha $ is a plane curve or $\dfrac{\tau }{\kappa }$ is a linear
function.
\end{theorem}

\begin{proof}
\textit{i)} From Eq.(3.7), we have%
\begin{equation*}
\left \langle U,K_{ss}\right \rangle =0\Longleftrightarrow \pm \kappa (1+v(%
\dfrac{\tau }{\kappa })^{^{\prime }})=0.
\end{equation*}%
Since $\kappa \neq 0$, we get $\dfrac{\tau }{\kappa }=-\dfrac{1}{v}s+c$,
where $c$ is a constant.\newline
\textit{ii) }$s-$parameter curves are also geodesic curves if and only if%
\begin{equation*}
U\times K_{ss}=\pm v(\dfrac{\tau }{\kappa })^{^{^{\prime \prime }}}B=0.
\end{equation*}%
Then, it follows that $v=0$ or $(\dfrac{\tau }{\kappa })^{^{^{\prime \prime
}}}=0$, i.e., the $s_{0}-$ parameter curve is the base curve $K(s,0)=\alpha
(s)$ or $\dfrac{\tau }{\kappa }=-as+c$, where $a$ and $c$ are constants.%
\newline
\textit{iii) Since }$U\times K_{vv}=0$ and $\left \langle
U,K_{vv}\right
\rangle =0$, $v-$parameter curves are both geodesic curves
and asymptotic curves, namely, $v-$parameter curves are straight lines.%
\newline
\textit{iv) We have }$f=0$. Furthermore,%
\begin{equation*}
F=\dfrac{\tau }{\kappa }(1+v(\dfrac{\tau }{\kappa })^{^{\prime
}})=0\Longleftrightarrow \tau =0\text{ \  \  \ or \  \  \ }\dfrac{\tau }{\kappa }%
=-\dfrac{1}{v}s+c,
\end{equation*}%
where $c$ is a constant. Then, the base curve $\alpha $ is a plane curve or $%
\dfrac{\tau }{\kappa }$ is a linear function.
\end{proof}

\begin{theorem}
Let $K(s,v)$ be the tangential Darboux developable surface of $\alpha $.
Then, the following expressions are satisfied for the parameter curves of $%
K(s,v)$.\newline
\textit{i)} $s-$parameter curves are also asymptotic curves if and only if
the base curve $\alpha $ is a general helix or a slant helix.\newline
ii) $s-$parameter curves are also geodesic curves if and only if the base
curve $\alpha $ is a plane curve or a slant helix.\newline
iii) $v-$parameter curves are straight lines.\newline
iv) The parameter curves are also lines of curvature.
\end{theorem}

\begin{proof}
\textit{i)} From Eq.(3.8), we have%
\begin{equation*}
\left \langle U,K_{ss}\right \rangle =0\Longleftrightarrow \pm (\kappa
^{2}+\tau ^{2})\sigma (s)(v-\sigma (s))=0.
\end{equation*}%
In the last equation, $\kappa $ and $\tau $ cannot be zero at the same time.
Hence, it concludes that $\sigma (s)=0$ or $\sigma (s)=v=$ $constant$ in
other words the base curve $\alpha $ is a general helix or a slant helix.%
\newline
\textit{ii) }$s-$parameter curves are also geodesic curves if and only if%
\begin{equation*}
U\times K_{ss}=\mp \kappa (v-\sigma (s))\sqrt{\kappa ^{2}+\tau ^{2}}T\pm
(\kappa ^{2}+\tau ^{2})(\frac{\sigma (s)-v}{\sqrt{\kappa ^{2}+\tau ^{2}}}%
)^{^{\prime }}N\pm \tau (v-\sigma (s))\sqrt{\kappa ^{2}+\tau ^{2}}B=0.
\end{equation*}%
Since $T$, $N$, and $B$ are linearly independent, we obtain%
\begin{eqnarray*}
\kappa (v-\sigma (s))\sqrt{\kappa ^{2}+\tau ^{2}} &=&0, \\
(\kappa ^{2}+\tau ^{2})(\frac{\sigma (s)-v}{\sqrt{\kappa ^{2}+\tau ^{2}}}%
)^{^{\prime }} &=&0, \\
\tau (v-\sigma (s))\sqrt{\kappa ^{2}+\tau ^{2}} &=&0.
\end{eqnarray*}%
From this, we get $\tau =0$ or $\sigma (s)=v=$ $constant$, i.e., the base
curve $\alpha $ is a plane curve or a slant helix.\newline
\textit{iii) Since }$U\times K_{vv}=0$ and $\left \langle U,K_{vv}\right
\rangle =0$, $v-$parameter curves are both geodesic curves and asymptotic
curves, namely, $v-$parameter curves are straight lines.\newline
\textit{iv) We have }$F=f=0$. Consequently, the parameter curves are also
lines of curvature.
\end{proof}


\begin{thebibliography}{9}
\bibitem{do Carmo} M.P. do Carmo, Differential Geometry of Curves and
Surfaces, Prentice-Hall, 1976.

\bibitem{dogan} F. Dogan, Isophote curves on timelike surfaces in Minkowski
3-Space, An. Stiint. Univ. Al. I. Cuza Iasi. Mat. (S.N.), DOI:
10.2478/aicu-2014-0020.

\bibitem{dogan} F. Dogan, Y.Yayl\i , On isophote curves and their
characterizations, Turkish J. Math., DOI: 10.3906/mat-1410-4.

\bibitem{izumiya} S. Izumiya, N. Takeuchi, Special curves and ruled
surfaces, Beitr. Algebra Geom. 44 (1) (2003) 203-212.

\bibitem{izumiya} S. Izumiya, N. Takeuchi, New special curves and
developable surfaces, Turkish J. Math. 28 (2004) 153-163.

\bibitem{leite} M.L. Leite, Surfaces with planar lines of curvature and
orthogonal systems of cycles, J. Math. Anal. Appl. 421 (2015) 1254-1273.

\bibitem{lucas} P. Lucas, J.A. Ortega-Yag\"{u}es, Rectifying curves in the
three-dimensional sphere, J. Math. Anal. Appl. 421 (2015) 1855-1868.

\bibitem{theisel} H. Theisel, G. Farin, The curvature of characteristic
curves on surfaces, Computer Graphics and Applications, IEEE 17 (6) (1997)
88-96.
\end{thebibliography}
\end{document}